\documentclass[a4paper,11pt]{amsart}
\usepackage{amsmath,amssymb,amsthm}
\usepackage{amscd}
\usepackage{array,enumerate}
\newtheorem*{Thm}{Theorem}
\newtheorem*{Prop}{Proposition}
\theoremstyle{definition}
\newtheorem*{Rem}{Remark}
\newcommand{\Cs}{\mbox{${\rm C}^\ast$}}

\title{Elementary constructions of non-discrete \Cs -simple groups}
\author{Yuhei Suzuki}
\address{Graduate School of Science,
Chiba University, Inage-ku, Chiba 263-8522, Japan}
\email{suzukiyu@ms.u-tokyo.ac.jp}
\begin{document}
\begin{abstract}
Recently Raum has given the first examples of locally compact non-discrete groups
with the simple reduced group \Cs -algebra, answering a question of de la Harpe.
Here we construct such groups
whose proof relies only on results in the discrete case.
\end{abstract} 
\maketitle
Let $G$ be a locally compact group.
Recall that the left regular representation $\lambda$ of $G$ is the unitary representation
of $G$ on the Hilbert space $L^2(G, \mu)$ acting by the left translation.
Here $\mu$ denotes the left Haar measure of $G$.
This representation induces a
$\ast$-representation of the group algebra $C_c(G)$ on  $L^2(G, \mu)$.
The reduced group \Cs -algebra ${\rm C}^\ast_\lambda(G)$ of $G$ is the operator norm closure of the image of $C_c(G)$ under this representation.
This provides basic and important examples of \Cs -algebras.

A locally compact group is said to be \Cs -simple if its reduced group \Cs -algebra has no proper closed two-sided ideal.
A basic question asks when a given group is \Cs -simple.
The first such a group was given by Powers \cite{Pow}, by showing that
the free groups are \Cs -simple.
His strategy is quite powerful, and until the recent breakthrough result of Kalantar and Kennedy \cite{KK},
his method was basically the only way to show \Cs -simplicity.
Now in the discrete case, fairly satisfactory characterizations of \Cs -simplicity are obtained \cite{Haa}, \cite{KK}, \cite{Ken}.

Recently, among other things,
Raum \cite{Rau} has constructed the first examples of non-discrete \Cs -simple groups,
based on properties of groups acting on trees.
The existence of such a group was asked by de la Harpe (\cite{Har}, Question 5).

In this paper, we establish a quite elementary method to construct
non-discrete \Cs -simple groups.
Our result gives explicit examples, and we only use previously known results in the discrete case.
(In fact, Powers's original result is enough to construct such a group.)
Furthermore, we show that these groups have the unique trace property.
Here we say a locally compact group has the unique trace property
if its reduced group \Cs -algebra has a unique lower semicontinuous semifinite trace up to scaling.

For a compact open subgroup $K$ of a locally compact group $G$,
let $p_K$ denote the image of the normalized characteristic function $\mu(K)^{-1}\chi_K$ in ${\rm C}^\ast_{\lambda}(G)$.
Then $p_K$ is the projection onto the subspace $L^2(G)^K$
of $K$-invariant functions.

To provide non-discrete \Cs -simple groups,
we first establish a criterion for \Cs -simplicity.
\begin{Prop}
Let $G$ be a locally compact group.
Assume we have a decreasing sequence $(K_n)_{n=1}^\infty$ of compact open subgroups of $G$
and an increasing sequence $(L_n)_{n=1}^\infty$ of clopen subgroups of $G$ with the following properties.
\begin{itemize}
\item Each $L_n$ contains $K_n$ and normalizes it.
\item The quotient groups $L_n/K_n$ are \Cs -simple.
\item The intersection $\bigcap_{n=1}^\infty K_n$ is the trivial subgroup $\{e\}$.
\item The union $\bigcup_{n=1}^\infty L_n$ is equal to $G$.
\end{itemize}
Then $G$ is \Cs-simple and has the unique trace property.

\end{Prop}
\begin{proof}
We only prove simplicity.
The unicity of trace is similarly shown.
(We recall that for discrete groups, \Cs -simplicity implies the unicity of tracial states on the reduced group \Cs -algebra \cite{BKKO}, \cite{Haa}.)
For each $n$,
let $A_n$ denote the \Cs -subalgebra of ${\rm C}^\ast_{\lambda}(G)$ generated by the set
$\{ \lambda_g p_{K_n}\colon g\in L_n\}$.
Since $L_n$ normalizes $K_n$ and $K_n$ is compact open,
the map $g\in L_n \mapsto \lambda_g p_{K_n} \in A_n$ 
defines a unitary representation $\pi$ of $L_n/K_n$ on $L^2(G)^{K_n}$.
The canonical isomorphism
$L^2(G)^{K_n}\cong \ell^2(K_n\backslash G)$ yields
the unitary equivalence of the left translation $L_n/K_n$-actions.
Since the left translation action of $L_n/K_n$ on $K_n\backslash G$ is free,
this shows the unitary equivalence of $\pi$ and a multiple of the left regular representation.
Hence $A_n$ is isomorphic to ${\rm C}^\ast_{\lambda}(L_n/K_n)$,
which is simple by our assumption.
The inclusions $K_{n+1}\subset K_n \subset L_n \subset L_{n+1}$
imply $A_n\subset A_{n+1}$.
Hence the norm closure $A$ of the increasing union $\bigcup_{n=1}^\infty A_n$ is simple.
Since $(K_n)_{n=1}^\infty$ decreases to the trivial subgroup
and $(L_n)_{n=1}^\infty$ increases to $G$,
we conclude $A={\rm C}^\ast_{\lambda}(G)$.
\end{proof}
\begin{Thm}
For each $n\in \mathbb{N}$,
let $\Gamma_n$ be a discrete group and let $F_n$ be a finite group acting on the group $\Gamma_n$ whose
semidirect product $\Gamma_n \rtimes F_n$ is \Cs -simple.
Set $G:=(\bigoplus_{n=1}^\infty \Gamma_n)\rtimes \prod_{n=1}^\infty F_n$.
Here $\bigoplus_{n=1}^\infty \Gamma_n$ is regarded as a discrete group, $\prod_{n=1}^\infty F_n$ is the compact group equipped with the product topology,
and the action $\prod_{n=1}^\infty F_n \curvearrowright \bigoplus_{n=1}^\infty \Gamma_n$ is the product
of given actions.
Then $G$ is \Cs-simple and has the unique trace property.

\end{Thm}
Before the proof, we note that Powers's result \cite{Pow} already gives groups and actions satisfying the conditions in Theorem.
For instance, consider the free product $\mathbb{Z}_2\ast \mathbb{Z}$.
Let $\Gamma$ be the kernel of the quotient homomorphism $\mathbb{Z}_2\ast \mathbb{Z}\rightarrow \mathbb{Z}_2$
given by sending the second free product component to $0$.
Then, since this map has a homomorphism lifting, we have a semidirect product decomposition
$\mathbb{Z}_2\ast \mathbb{Z} = \Gamma \rtimes \mathbb{Z}_2$.
Powers's proof \cite{Pow} shows that ${\rm C}^\ast_\lambda(\mathbb{Z}_2\ast \mathbb{Z})$ is simple and has a unique tracial state.
\begin{proof}[Proof of Theorem]
For each $n\in \mathbb{N}$, set $K_n:=\prod_{k=n+1}^\infty F_k$
and $L_n:=(\bigoplus_{k=1}^n \Gamma_k) \rtimes \prod_{k=1}^\infty F_k$.
(Both are regarded as a clopen subgroup of $G$ in the canonical way.)
Then it is not hard to check that the sequences $(K_n)_{n=1}^\infty$ and $(L_n)_{n=1}^\infty$ satisfy
the conditions in Proposition.
\end{proof}
\subsection*{Group von Neumann algebras}
For a group as in Theorem, in a similar way,
it can be shown that its group von Neumann algebra is a factor of type ${\rm I\hspace{-.1em}I}_\infty$.
This factor is not injective, as a finite corner has a non-injective subfactor.
By modifying our construction, for any rational number $0<q\leq 1$, we can construct a \Cs -simple group
whose group von Neumann algebra is a factor of type ${\rm I\hspace{-.1em}I\hspace{-.1em}I}_q$ as follows.
\begin{proof}[Sketch of the construction]
We only show the case $q<1$. The case $q=1$ then follows by taking an appropriate direct product of such groups (cf. \cite{Con}).
Take $n_1, n_2\in \mathbb{N}$ with $q=n_1/n_2$.
For $i=1, 2$, let $\Gamma_i$ be a discrete group and
let $F_i$ be a finite group of order $n_i$ acting on $\Gamma_i$ such that the semidirect product $\Gamma_i \rtimes F_i$
is \Cs -simple. (Such ones are easily found in a similar way to that in the remark below Theorem.)
Set $H_i:=(\bigoplus_{n\in \mathbb{Z}} \Gamma_i)\rtimes ((\bigoplus_{n\leq 0} F_i)\times (\prod_{n\geq 1} F_i))$ for $i=1, 2$. Then $H_i$ is naturally regarded as a locally compact group.
It is easy to check that $H_i$ satisfies the conditions in Proposition.
Now let $\alpha_i$ be the automorphism of $H_i$ given by the forward shift of the indices $n \in \mathbb{Z}$.
Set $H:=H_1 \times H_2$.
Then $\alpha:=\alpha_1 \times \alpha_2^{-1}$ defines an automorphism of $H$.
Put $G:=H\rtimes _\alpha \mathbb{Z}$.
Obviously $\alpha$ induces automorphisms of ${\rm C}^\ast_\lambda(H)$ and $L(H)$. We denote them by the same symbol $\alpha$.
Then we have isomorphisms
${\rm C}^\ast_\lambda(G)\cong {\rm C}^\ast_\lambda(H)\rtimes_\alpha \mathbb{Z}$ and
$L(G)\cong L(H)\bar{\rtimes}_\alpha \mathbb{Z}$.
Since $L(H)$ is a type ${\rm I\hspace{-.1em}I}_\infty$ factor
and $\alpha$ scales the trace on $L(H)$ at the rate $q$,
Connes's theorem (\cite{Con} Theorem 4.4.1) shows that $L(G)$ is a type ${\rm I\hspace{-.1em}I\hspace{-.1em}I}_q$ factor.
Also, since $\alpha$ scales a semifinite trace on ${\rm C}^\ast_\lambda(H)$,
nonzero powers of $\alpha$ are outer.
Now Kishimoto's theorem (\cite{Kis}, Theorem 3.1) shows the \Cs -simplicity of $G$.
\end{proof}
\begin{Rem}
Raum \cite{Rau} has already constructed \Cs -simple groups whose group von Neumann algebras are factors of types as above.
Also, for any $\lambda\in [0, 1]$, Sutherland \cite{Sut} has constructed a factorial group von Neumann algebra of type ${\rm I\hspace{-.1em}I\hspace{-.1em}I}_\lambda$. 
\end{Rem}
\subsection*{Acknowledgement}
This work was carried out while the author was staying at Mittag-Leffler institute for the program ``Classification of operator algebras: complexity, rigidity, and dynamics''.
He acknowledges the organizers and the institute for the invitation and the kind hospitality.
He also thanks Sven Raum for helpful comments on the first draft and letting him know the reference \cite{Sut}.
This work was supported by JSPS Research Fellowships for Young Scientists (PD 28-4705).

\end{document}